\documentclass[12pt,reqno]{amsart}

\usepackage{amssymb, amsthm, amsmath}
\usepackage[pdftex]{graphicx}
\usepackage[single]{accents}

\newtheorem{theorem}{Theorem}[section]
\newtheorem{lemma}[theorem]{Lemma}
\newtheorem{prop}[theorem]{Proposition}
\newtheorem{corollary}[theorem]{Corollary}

\theoremstyle{definition}

\theoremstyle{remark}
\newtheorem{remark}[theorem]{Remark}

\newcommand{\mcc}{{\raise 0.45ex \hbox{c}}}

\numberwithin{equation}{section}

\newcommand{\D}{\mathbb{D}}

\newcommand{\C}{\mathbb{C}}
\newcommand{\T}{\mathbb{T}}

\newcommand{\ip}[2]{\langle #1, #2 \rangle}
\newcommand{\mcH}{\mathcal{H}}
\newcommand{\mcO}{\mathcal{O}}

\newcommand{\tl}[1]{\accentset{\leftarrow}{#1}}

\newcommand{\clspan}{\overline{\mathrm{span}}}
\newcommand{\Span}{\mathrm{span}}
\newcommand{\bshilb}{L^2(1/|p_{n,m}|^2 d\sigma)}

\title[Bivariate Bernstein-Szeg\H{o} measures]{Orthogonality
  relations for bivariate Bernstein-Szeg\H{o} measures}
\date{November 22, 2011} 
\author[J.~Geronimo]{Jeffrey S. Geronimo}
\thanks{JSG is supported in part by Simons Foundation Grant \#210169.}
\address{JSG, School of Mathematics, Georgia Institute of Technology,
Atlanta, GA 30332--0160, USA}
\email{geronimo@math.gatech.edu}
\author[P.~Iliev]{Plamen Iliev}
\thanks{PI is supported in part by NSF Grant \#0901092}
\address{PI, School of Mathematics, Georgia Institute of Technology,
Atlanta, GA 30332--0160, USA}
\email{iliev@math.gatech.edu}
\author[G.~Knese]{Greg Knese}
\thanks{GK is supported in part by NSF Grant \#1048775}
\address{GK, Department of Mathematics, University of Alabama, Box 870350, Tuscaloosa, AL 35487-0350, USA}
\email{geknese@bama.ua.edu}
\dedicatory{To Francisco~Marcell\'an on the occasion of his 60th birthday}
\begin{document}

\begin{abstract}
The orthogonality properties of certain subspaces associated with bivariate 
Bernstein-Szeg\H o measures are considered. It is shown that these spaces 
satisfy more orthogonality relations than expected from the relations that 
define them. The results are used to prove a Christoffel-Darboux like formula 
for these measures.
\end{abstract}
\maketitle
\section{Introduction}
In the study of bivariate polynomials orthogonal on the bi-circle
progress has recently been made in understanding these polynomials in
the case when the orthogonality measure is purely absolutely
continuous with respect to Lebesgue measure of the form
$$d\mu=\frac{d\sigma}{|p_{n,m}(e^{i\theta},e^{i\phi})|^2},$$ 
where $p_{n,m}(z,w)$ is of degree $n$ in $z$ and $m$ in $w$ and is
stable i.e. is nonzero for $|z|,|w|\le 1$ and $d\sigma$ is the normalized Lebesgue measure
on the torus $\T^2$.  Such measures have come to
be called Bernstein-Szeg\H o measures and they played an important
role in the extension of the Fej\'er-Riesz factorization lemma to two
variables \cite{GWa}, \cite{GWb}, \cite{Ka}, \cite{Kb}.  In particular
in order to determine whether a positive trigonometric polynomial can
be factored as a magnitude square of a stable polynomial an important
role was played by a bivariate analog of the Christoffel-Darboux
formula. The derivation of this formula was non trivial even if one
begins with the stable polynomial $p_{n,m}$, \cite{GWa}, \cite{GKVW},
\cite{Ka}, \cite{Wo}.  This formula was shown to be a special case of
the formula derived by Cole and Wermer \cite{Ka} through operator
theoretic methods. Here we give an alternative derivation of the
Christoffel-Darboux formula beginning with the stable polynomial
$p_{n,m}$. This is accomplished by examining the orthogonality
properties of the polynomial $p_{n,m}$ in the space
$L^{2}(d\mu)$. These orthogonality properties imbue certain subspaces
of $L^{2}(d\mu)$ with many more orthogonality relations than would
appear by just examining the defining relations for these spaces.

We proceed as follows. In section 2 we introduce the notation to be
used throughout the paper and examine the orthogonality properties of
the stable polynomial $p_{n,m}$ in the space $L^{2}(d\mu)$.  We also
list the properties of a sequence of polynomials closely associated
with $p_{n,m}$. In section 3 we state, and in section 4, prove, one of
the main results of the paper on the orthogonality of certain
subspaces of $L^{2}(d\mu)$. We also establish several follow-up
results which are then used in section 5 to derive the
Christoffel-Darboux formula. The proof is reminiscent of that given in
\cite{GKVW} and \cite{Kc}. In section 6, we study connections to the
parametric moment problem.

\section{Preliminaries}
Let $p_{n,m} \in \C[z,w]$ be stable with degree $n$ in $z$ and $m$ in
$w$.  We will frequently use the following partial
order on pairs of integers:
\[
(k,l) \leq (i,j) \text{ iff } k \leq i \text{ and } l\le j.
\]
The notations $\nleq, \ngeq$ refer to the negations of the above
partial order.
Define 
\[
\tl{p}_{n,m}(z,w) = z^nw^m \overline{p_{n,m}(1/\bar{z},1/\bar{w})}.
\]

When we refer to ``orthogonalities,'' we shall always mean
orthogonalities in the inner product $\ip{\cdot}{\cdot}$ of the
Hilbert space $\bshilb$ on $\T^2$.  Notice that $\bshilb$ is
topologically isomorphic to $L^2(\T^2)$ but we use the different
geometry to study $p_{n,m}$.

The polynomial $p_{m,n}$ is orthogonal to more monomials than the one
variable theory might initially suggest. More precisely,

\begin{lemma} \label{porth} 
In $L^2(1/|p_{n,m}|^2 d\sigma)$, $p_{n,m}$ is orthogonal to the set
\[
\{ z^i w^j: (i,j) \nleq (0,0) \}
\]
 and $\tl{p}_{n,m}$ is orthogonal
to the set
\[
\{ z^i w^j : (i,j) \ngeq (n,m) \}.
\]
\end{lemma}

\begin{proof} Observe that since $1/p_{n,m}$ is holomorphic in $\overline{\D^2}$
\[
\begin{aligned}
\ip{z^i w^j}{p_{n,m}} &= \int_{\T^{2}} z^i w^j \overline{p_{n,m}(z,w)}
\frac{d\sigma}{|p_{n,m}(z,w)|^2} \\ &= \int_{\T^{2}} \frac{z^i w^j}{p_{n,m}(z,w)}
d\sigma = 0 \text{ if } (i,j) \nleq (0,0)
\end{aligned}
\]
by the mean value property (either integrating first with respect to
$z$ or $w$ depending on whether $i>0$ or $j>0$).  The claim about
$\tl{p}_{n,m}$ follows from the observation $\ip{z^iw^j}{\tl{p}_{n,m}}
= \ip{p_{n,m}}{z^{n-i}w^{m-j}}$.
\end{proof}

\begin{figure} 
\caption{Orthogonalities of $p_{n,m}$}
\scalebox{0.6}{\includegraphics{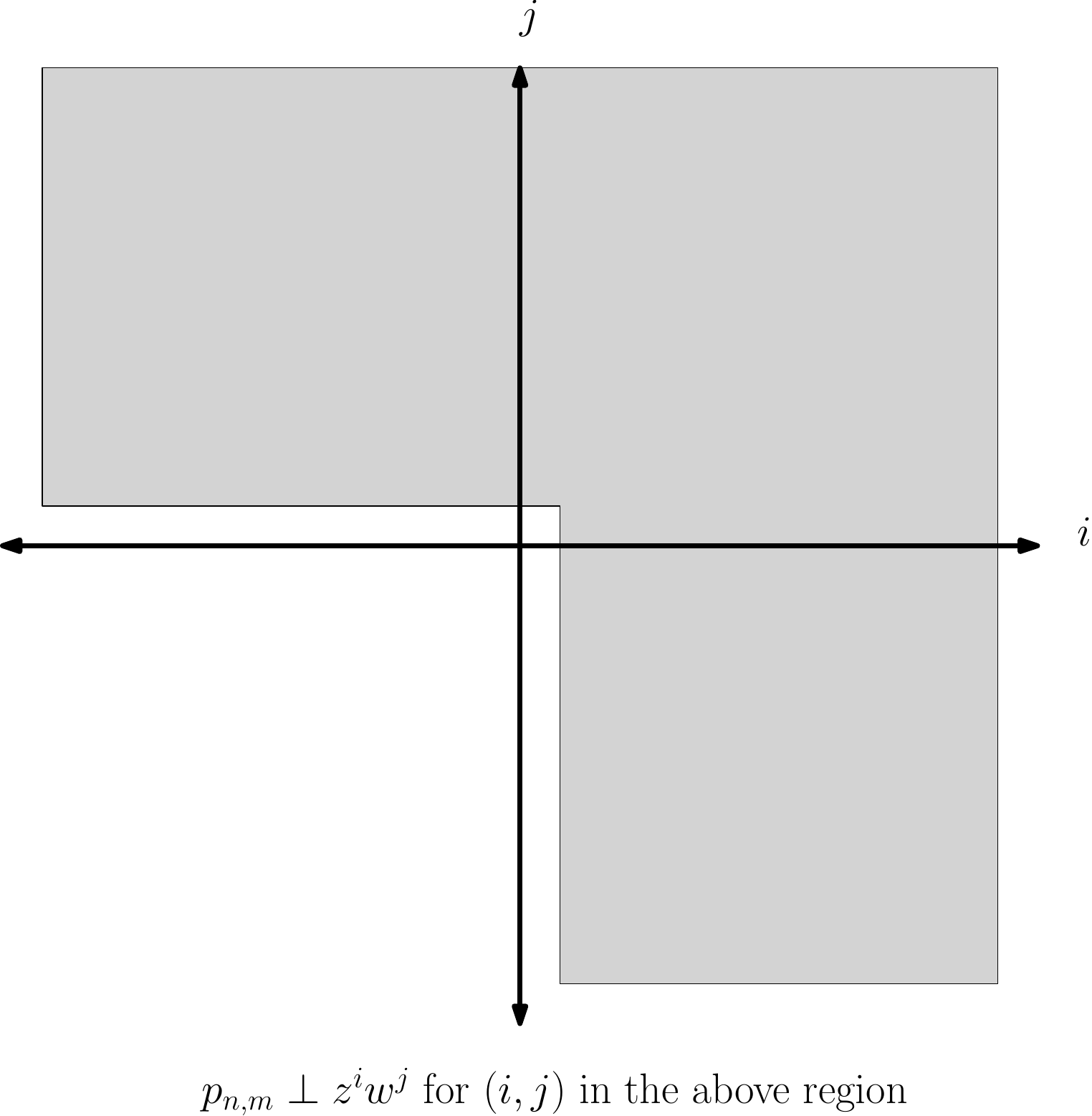}} \label{orthsetp}
\end{figure}

Write  $p_{n,m}(z,w)=\sum^m_{i=0}p_i(z)w^i$.

Since $p_{n,m}(z,w)$ is
stable  it follows from the Schur-Cohn test for stability \cite{GWa} that the $m\times m$ matrix
\begin{align}\label{gohsme}
T_m(z)&=
\begin{bmatrix}p_0(z) && \bigcirc\\
p_1(z) &\ddots\\
\vdots \\
p_{m-1}(z) &\cdots & p_0(z)\end{bmatrix}
\begin{bmatrix}\bar p_0(1/z) & \bar p_1(1/z) &\cdots & \bar p_{m-1}(1/z)\\\nonumber
&\ddots \\
\bigcirc &\cdots&\cdots& \bar p_0 (1/z)\end{bmatrix}\\
&\qquad -
\begin{bmatrix}\bar p_m(1/z) && \bigcirc\\
\vdots &\ddots\\
\bar p_1(1/z) &\cdots & \bar p_m(1/z)\end{bmatrix}
\begin{bmatrix} p_m(z) &\cdots & p_1(z)\\
\vdots &\ddots\\
\bigcirc&& p_m(z)\end{bmatrix}
\end{align}
is positive definite for $|z|=1$. Here $\bar{p}_j(z) =
\overline{p_j(\bar{z})}$.  

Define the following parametrized version of a one variable
Christoffel-Darboux kernel
\begin{align}\label{christo}
L(z,w;\eta) &= z^n\frac{p_{n,m}(z,w)\overline{p_{n,m}(1/\bar{z},\eta)} - \tl{p}_{n,m}(z,w)\overline{\tl{p}_{n,m}(1/\bar{z},\eta)}}{1-w\bar{\eta}}
\\\nonumber & =z^n[1, \ldots,w^{m-1}] T_m(z)[1,\ldots,\eta^{m-1}]^{\dag}\\\nonumber &=\sum_{j=0}^{m-1} a_j(z,w)\bar \eta^j,
\end{align}
where $a_j(z,w),\ j=0,\ldots, m-1$ are polynomials in $(z,w)$, as the
following lemma shows in addition to several other important
observations. 
\begin{lemma}\label{chrisdar}
Let $p_{n,m}(z,w)$ be a stable polynomial of degree $(n,m)$. Then,
\begin{enumerate} 
\item $L$ is a polynomial of degree $(2n,m-1)$ in $(z,w)$ and a polynomial
of degree $m-1$ in $\bar{\eta}$.

\item $L(\cdot,\cdot;\eta)$ spans a subspace of dimension $m$ as
  $\eta$ varies over $\C$.

\item $L$ is symmetric in the sense that
\[
L(z,w;\eta)= z^{2n} (w \bar{\eta})^{m-1} \overline{L(1/\bar{z}, 1/\bar{w}; 1/\bar{\eta})},
\]
so $a_k=\tl{a}_{m-k-1}$.

\item $L$ can be written as
\[
L(z,w;\eta) = p_{n,m}(z,w)A(z,w;\eta) + \tl{p}_{n,m}(z,w)B(z,w;\eta)
\]
where $A,B$ are polynomials of degree $(n,m-1, m-1)$ in
$(z,w, \bar{\eta})$.
\end{enumerate}

\begin{proof} The numerator of $L$ vanishes when $w = 1/\bar{\eta}$,
  so the factor $(1-w\bar{\eta})$ divides the numerator.  This gives
  (1).

  For (2), when $|z| = 1$ use equation~\eqref{christo}. 
Since
$T_m(z) >0$ for $|z|=1$, $L(z,w;\eta)$ spans a set of polynomials of
dimension $m$.  

For (3), this is just a computation.  

For (4), observe that (suppressing the dependence of $p$ on $n$ and $m$),
\begin{align}
 & z^n \frac{p(z,w)\overline{p(1/\bar{z},\eta)}  
 - \tl{p}(z,w)\overline{\tl{p}(1/\bar{z},\eta)}}{1-w\bar{\eta}}
  \nonumber \\
  &= p(z,w) \underset{A(z,w;\eta)}{\underbrace{\left(\frac{\bar{\eta}^m\tl{p}(z,1/\bar{\eta})
        - \bar{\eta}^m\tl{p}(z,w)}{1 - w\bar{\eta} } \right)} } 
  + \tl{p}(z,w) \underset{B(z,w;\eta)}{\underbrace{ \left(\frac{\bar{\eta}^m p(z,w) -
      \bar{\eta}^mp(z,1/\bar{\eta})}{1 - w\bar{\eta}} \right)} }. \label{ABdef}
\end{align}
\end{proof}

\end{lemma}

\section{Orthogonality relations in $\bshilb$}

Our main goal is to prove that $L$ and $a_0, \dots, a_{m-1}$ possess a
great many orthogonality relations in $L^2(1/|p_{n,m}|^2 d\sigma)$.
The orthogonality relations of $L$ are depicted in Figure
\ref{orthset_L}.

\begin{figure}
\scalebox{0.6}{\includegraphics{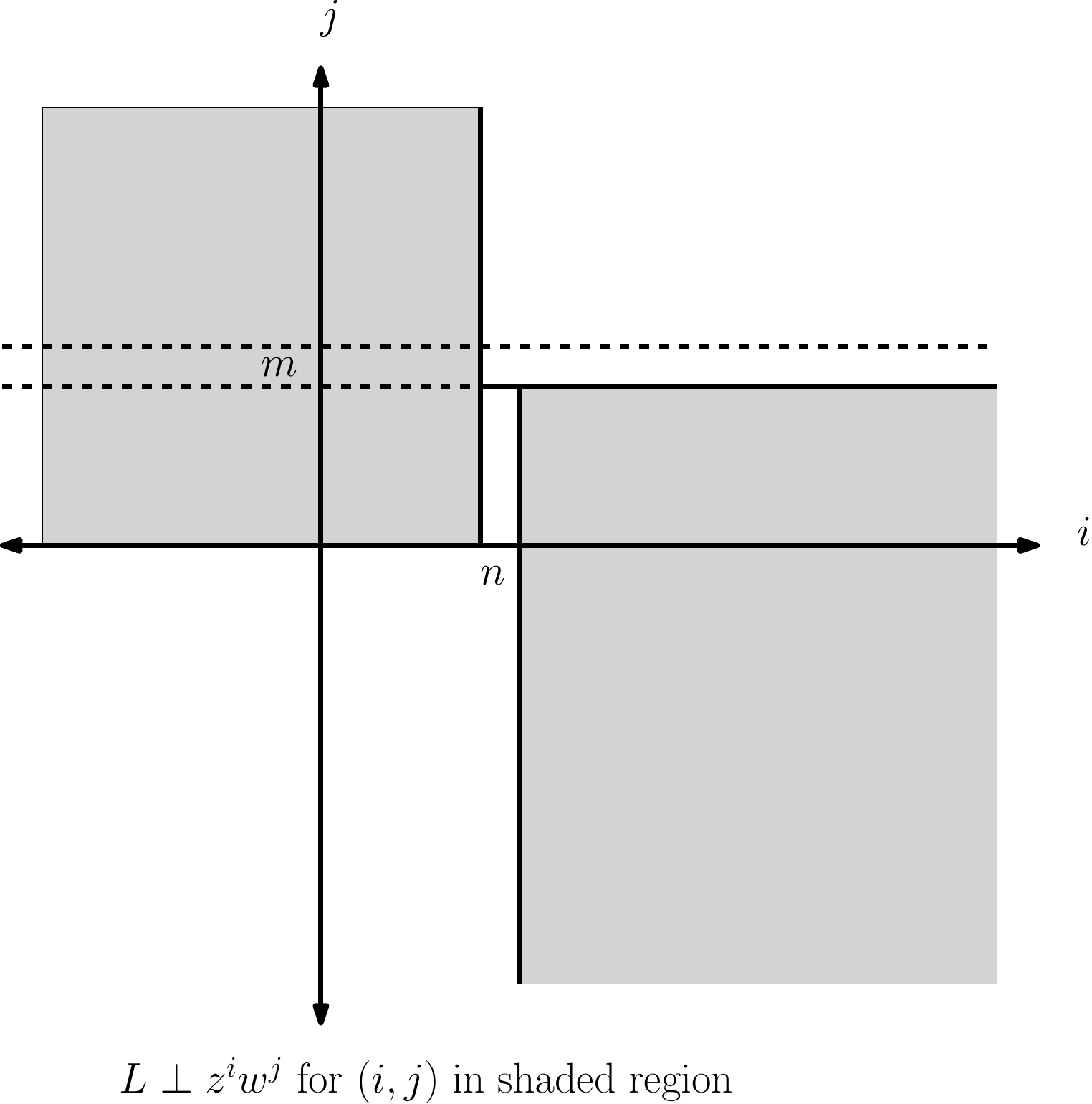}}
\caption{Orthogonalities of $L$. See Theorem \ref{mainprop}}  \label{orthset_L}
\end{figure}

\begin{theorem} \label{mainprop} 
 In $L^2(1/|p_{n,m}|^2 d\sigma)$, each $a_k$ is orthogonal to the set
\[
\begin{aligned}
\mcO_k = &\{z^i w^j: i > n, j < 0 \} \\
&\cup \{ z^i w^j: 0 \leq j < m, j \ne k\} \\
&\cup \{z^i w^j: i < n, j \geq m\} \\
&\cup \{z^i w^k: i \ne n \}.
\end{aligned}
\]

In $L^2(1/|p_{n,m}|^2 d\sigma)$, $L(\cdot,\cdot; \eta)$ is orthogonal to the
set
\begin{align}
\mcO =  &\{z^i w^j: i > n, j < 0 \} \nonumber \\
&\cup\{ z^i w^j: i \ne  n, 0 \leq j < m \} \label{Wdef} \\
&\cup \{z^i w^j: i < n, j \geq m\}. \nonumber
\end{align}
\end{theorem}

Note that

\begin{align*}
\mcO_k &= \{z^n w^j: 0\leq j < m, j \ne k\} \cup \mcO, \\
\mcO &= \bigcap_{k=0}^{m-1} \mcO_k.
\end{align*}

\begin{corollary} \label{cor:ak}
In $L^2(1/|p_{n,m}|^2 d\sigma)$, the polynomial $a_k$ is uniquely determined
(up to unimodular multiples) by the conditions:
\[
a_k \in \Span \{z^i w^j: (0,0) \leq (i,j) \leq
(2n,m-1)\},
\]
\[
\begin{aligned}
a_k \perp &\{z^i w^j: (0,0) \leq (i,j) \leq
(2n,m-1),  j \ne k\}\\
& \cup \{z^i w^k: 0 \leq i \leq 2n, i \ne n\},
\end{aligned}
\]
and
\[
||a_k||^2 = \int_{-\pi}^{\pi} T_{k,k}(e^{i\theta},e^{i\theta}) \frac{d\theta}{2\pi}.
\]
\end{corollary}

(The last fact follows from Proposition \ref{prop:Lnorm}, which is not
currently essential.)

\begin{figure}
\centerline{\scalebox{0.6}{\includegraphics{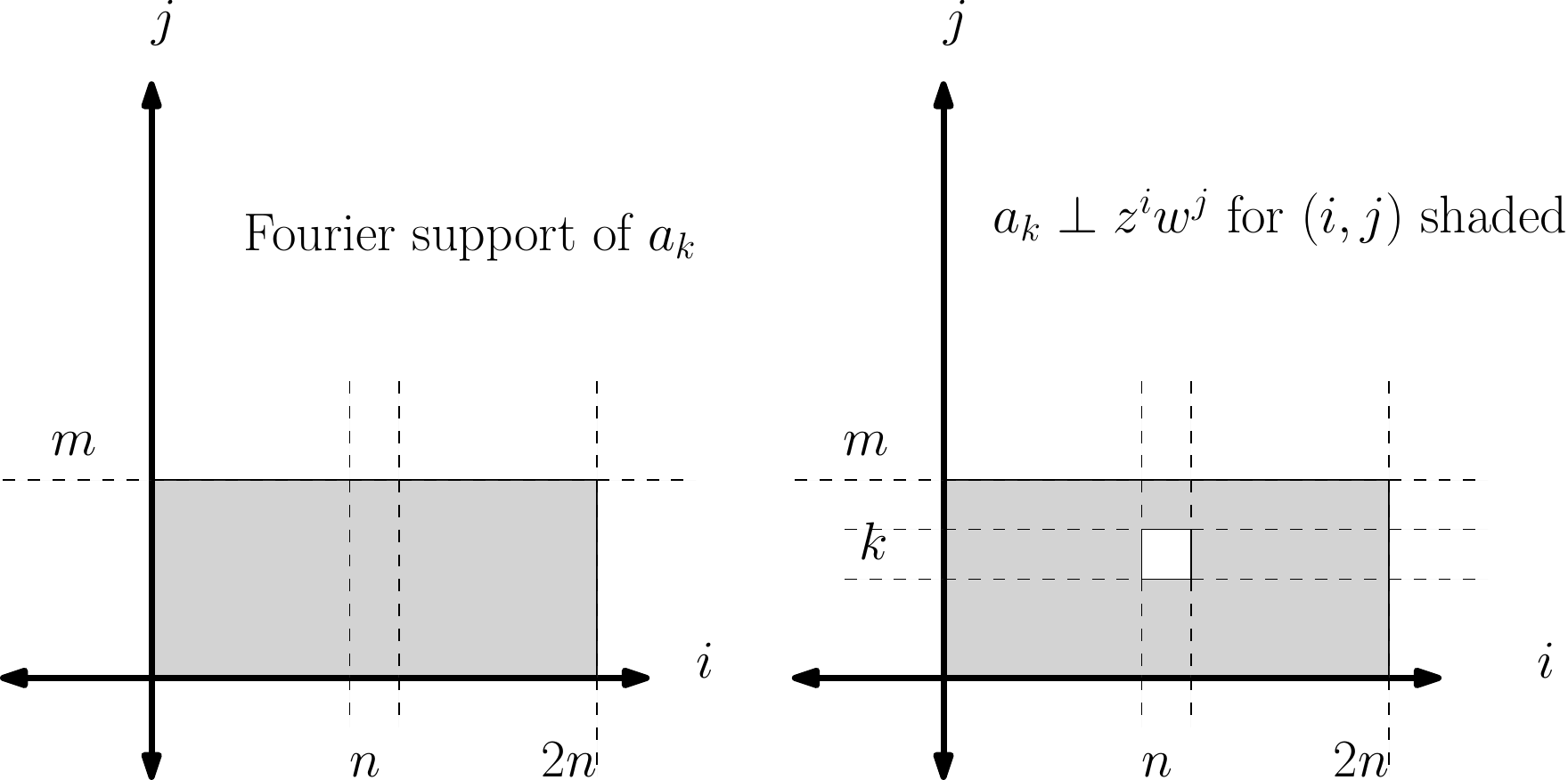}}}
\caption{The $a_k$ are uniquely determined by the above
  properties. See Corollary \ref{cor:ak}.}
\label{aksupporth}
\end{figure}

\begin{remark} \label{akremark} 
We emphasize that (1) each $a_k$ is explicitly given from coefficients
of $p_{n,m}$, (2) each $a_k$ is determined by the orthogonality
relations in Corollary \ref{cor:ak} (depicted in Figure
\ref{aksupporth}), and (3) each satisfies the additional orthogonality
relations from Theorem \ref{mainprop}.  One useful consequence of this
is that the set
\[
\{z^j a_k(z,w): j \in \mathbb{Z}, 0\leq k < m\}
\]
is dual to the monomials
\[
\{z^{j+n} w^k: j \in \mathbb{Z}, 0\leq k < m\}
\]
within in the subspace
\[
S = \clspan \{z^j w^k: j \in \mathbb{Z}, 0\leq k < m\}.
\]
Namely,
\[
\ip{z^{j_1+n}w^{k_1}}{z^{j_2} a_{k_2}} = 0
\]
unless $j_1=j_2$ and $k_1=k_2$.

In particular, if $f \in S$, then 
\begin{equation} \label{dualbasis}
f \perp z^ja_k \text{ implies } \hat{f}(j+n,k) = 0.
\end{equation}
\end{remark}

\section{The proof of Theorem \ref{mainprop}}
We begin by writing

\[
A(z,w;\eta) = \sum_{j=0}^{m-1} A_j(z,w)\bar{\eta}^{j} \qquad
B(z,w;\eta) = \sum_{j=0}^{m-1} B_j(z,w) \bar{\eta}^{j}.
\]

Recall equation \eqref{christo} and Lemma \ref{chrisdar} item (4).  By
examining coefficients of $\bar{\eta}^j$ in $L$
\[
a_j = p_{n,m} A_j+ \tl{p}_{n,m} B_j.
\]
Also, $A_j$ and $B_j$ have at most degree $j$ in $w$.  To see this,
recall equation \eqref{ABdef} and observe that
\[
A(z,w;\eta) = \sum_{j} \tl{p}_j(z) \bar{\eta}^j \frac{1-
  (w\bar{\eta})^{m-j}}{1-w\bar{\eta}} 
\]
which shows that $A_j(z,w)$ has degree at most $j$ in $w$ (i.e. powers
of $w$ only occur next to greater powers of $\eta$).  The same holds
for $B$.

\begin{proof}[Proof of Theorem \ref{mainprop}]
By Lemma \ref{porth}, $p_{n,m}$ is orthogonal to
\[
\{ z^i w^j : (i,j) \nleq (0,0)\}
\]
and since $A_k$ has degree at most $n$ in $z$ and $k$ in $w$,
\[
p_{n,m} A_k \text{ is orthogonal to } \{z^i w^j: (i,j) \nleq (n,k)\}.
\]

Also,
\[
\tl{p}_{n,m} B_k \text{ is orthogonal to } \{z^i w^j : (i,j) \ngeq (n,m) \}
\]
since the orthogonality relation for $\tl{p}_{n,m}$ (also from Lemma
\ref{porth}) is unaffected by multiplication by holomorphic monomials.

Hence, $a_k = p_{n,m}A_k + \tl{p}_{n,m} B_k$ is orthogonal to the intersection of
these sets; namely,
\begin{equation} \label{aorth1}
\{z^{i} w^j: (i,j) \nleq (n,k) \text{ and } (i,j) \ngeq (n,m)\}.
\end{equation}
Since
\[
a_{m-k-1} \perp \{z^{i} w^j: (i,j) \nleq (n,m-k-1) \text{ and }
(i,j) \ngeq (n,m)\}
\]
and since $a_k = \tl{a}_{m-k-1} = z^{2n} w^{m-1}
\overline{a_{m-k-1}(1/\bar{z}, 1/\bar{w})}$,
\begin{align}
a_k \perp& \{z^{2n-i} w^{m-j-1}: (i,j) \nleq (n,m-k-1) \text{ and }
(i,j) \ngeq (n,m)\} \nonumber \\
& = \{z^{i} w^j: (n,k) \nleq (i, j) \text{ and } (n,-1) \ngeq
(i,j)\}. \label{aorth2}
\end{align}

Hence, $a_k$ is orthogonal to the union of the sets in \eqref{aorth1}
and \eqref{aorth2}.  The set in \eqref{aorth2} contains $\{z^iw^j:
i<n, j \geq 0\}$ and the set in \eqref{aorth1} contains $\{z^iw^j:
i>n, j \leq m-1\}$.  Also, the set in \eqref{aorth1} contains $\{z^nw^j:
k<j\leq m-1\}$ while the set in \eqref{aorth2} contains $\{z^nw^j:
0\leq j< k\}$.  Combining all of this we get $a_k \perp \mcO_k$.

Finally, $L$ is orthogonal to
the intersection of $\mcO_0, \dots, \mcO_{m-1}$.
\end{proof}

We now look at the space generated by shifting the $a_k$'s by powers
of $z$.

\begin{theorem}\label{GW} With respect to $L^2(\frac{d\sigma}{|p_{n,m}|^2})$,
\begin{align}
&\clspan\{z^i a_j(z,w): 0 \leq i, 0\leq j < m\} \nonumber \\
&= \clspan\{z^i w^j: 0 \leq i, 0\leq j < m\} \ominus \clspan\{z^i w^j:
  0 \leq i < n, 0\leq j < m\} \label{GWspace}
\end{align}
and this is orthogonal to the larger set
\[
\clspan \{z^i w^j: i < n, j \geq 0\}.
\]
\end{theorem}

\begin{proof}
Since the $a_k$ are polynomials of degree at most $m-1$ in $w$, it is
clear that
\[
\clspan\{z^i a_j(z,w): 0 \leq i, 0\leq j < m\} \subset \clspan\{z^i
w^j: 0 \leq i, 0\leq j < m\}.
\]
By Theorem \ref{mainprop}, the $a_k$ are orthogonal to the spaces
\[
\clspan \{z^i w^j: i < n, j \geq 0\} \supset \clspan\{z^i
w^j:  i < n, 0\leq j < m\},
\]
and since these spaces are invariant under multiplication by
$\bar{z}$, the polynomials $z^ia_k$ are also orthogonal to these
spaces for all $i\geq 0$.  So,
\[
\clspan\{z^i a_j(z,w): 0 \leq i, 0\leq j < m\} \perp \clspan \{z^i
w^j: i < n, j \geq 0\}.
\]
Therefore,
\begin{align}
&\clspan\{z^k a_j(z,w): 0 \leq k, 0\leq j < m\} \nonumber \\
&\subset \clspan\{z^i w^j: 0 \leq i, 0\leq j < m\} \ominus \clspan\{z^i w^j: 0 \le i < 
     n, 0\leq j < m \} \label{containment}
\end{align}
and this containment must in fact be an equality.  

Indeed, any $f$ in
\[
\clspan\{z^i w^j: 0 \leq i, 0\leq j < m\}
\]
which is orthogonal to $\{z^k a_j(z,w): 0 \leq k, 0\leq j < m\}$
satisfies $\hat{f}(i,j) = 0$ for $i \geq n$ and $0\leq j < m$ by
Remark \ref{akremark} and equation \eqref{dualbasis}.  Such an $f$ cannot also
be orthogonal to the space $\clspan\{z^i w^j: 0 \le i < n, 0\leq j < m
\}$ without being identically zero.
\end{proof}

Define
\[
\begin{aligned}\label{h}
H =& \Span\{z^i w^j: (0,0) \leq (i, j) \leq (n,m-1)\} \\
&\ominus \Span\{z^i w^j: (0,0) \leq (i, j) \leq (n-1,m-1)\}.
\end{aligned}
\]
Define also the reflection $\tl{H}$
\[
\begin{aligned}\label{hrev}
\tl{H} =& \Span\{z^i w^j: (0,0) \leq (i, j) \leq (n,m-1)\} \\
&\ominus \Span\{z^i w^j: (1,0) \leq (i, j) \leq (n,m-1)\}.
\end{aligned}
\]

\begin{prop}\label{reph}
We have the following orthogonal direct sum decompositions in $L^2(1/|p_{n,m}|^2
d\sigma)$
\begin{equation} \label{spaceH}
\clspan\{z^k a_j(z,w): 0 \leq k, 0\leq j < m\} =
\bigoplus_{i=0}^{\infty} z^i H
\end{equation}
\begin{equation}\label{spaceHrev}
\mcH_1 := \clspan\{z^k w^j: 0 \leq k, 0 \leq j < m\} = \bigoplus_{i=0}^{\infty}
z^i \tl{H}.
\end{equation}
If $K_H$ is the reproducing kernel for $H$ and $K_{\tl{H}}$ is the
reproducing kernel for $\tl{H}$, then the reproducing kernel for the
spaces in \eqref{spaceH} and \eqref{spaceHrev} are given by
\[
\frac{K_H(z,w;z_1,w_1)}{1-z\bar{z}_1} \text{ and }
\frac{K_{\tl{H}} (z,w;z_1,w_1)}{1-z\bar{z}_1}
\]
respectively.
\end{prop}

\begin{proof}
Now $H$ is an $m$ dimensional space of polynomials contained in the
space \eqref{GWspace} of the previous theorem.  In particular,
\begin{equation} \label{horth}
H \perp \clspan \{z^i w^j: i < n, j \geq 0\},
\end{equation}
and 
\[
H = \clspan\{z^iw^j: i\leq n, 0\leq j < m\} \ominus \clspan\{z^i w^j: i<n,
0\leq j <m\}
\]
since this space is also $m$ dimensional and contains $H$.  From this
it is clear that $H \perp z^iH$ for $i>0$ and we have
\[
\begin{aligned}
\bigoplus_{i=0}^{\infty} z^i H =& \clspan\{ z^iw^j: 0\leq j < m\} \\
&\ominus \clspan\{z^i w^j: i<n, 0 \leq j < m\}.
\end{aligned}
\]
Since shifts of $H$ are contained in $\clspan\{ z^iw^j: 0\leq i, 0\leq
j < m\}$, we must have
\[
\begin{aligned}
\bigoplus_{i=0}^{\infty} z^i H =& \clspan\{ z^iw^j: 0\leq i, 0\leq j < m\} \\
&\ominus \clspan\{z^i w^j: 0\leq i< n, 0 \leq j < m\}
\end{aligned}
\]
which combined with \eqref{GWspace} gives \eqref{spaceH}.

Next, $\tl{H}$ is also $m$ dimensional and by \eqref{horth} is
orthogonal to
\[
\{z^iw^j: i > 0; j < m\}
\]
which in particular contains the strip $\{z^i w^j: i > 0; 0 \leq j
< m\} = z \{z^i w^j: i \geq 0; 0 \leq j
< m\}$. So, 
\[
\begin{aligned}
\tl{H} =& \clspan \{ z^iw^j: 0\leq i, 0\leq j < m\} \\ 
&\ominus z\ \clspan \{ z^iw^j: 0\leq i, 0 \leq j < m\}
\end{aligned}
\]
by dimensional considerations.
Therefore,
\[
\mcH_1 = \bigoplus_{j\geq 0} z^j \tl{H}.
\]

The formulas for the reproducing kernels are direct consequences of
the orthogonal decompositions (see \cite{Ka} for more on this).
\end{proof}

\begin{lemma}
In $\bshilb$ the reproducing kernel for 
\[
\mcH = \clspan\{z^i w^j: (0,0) \leq (i, j) \ngeq (n,m)\}
\]
is
\[
\frac{p_{n,m}(z,w)\overline{p_{n,m}(z_1,w_1)} -
  \overleftarrow{p_{n,m}}(z,w)\overline{\overleftarrow{p_{n,m}}(z_1,w_1)}}
     {(1-z\bar z_1)(1-w \bar w_1)}.
\]
\end{lemma}

\begin{proof}
First,
\[
K(z,w;z_1,w_1) = K_{(z_1,w_1)}(z,w) = \frac{p_{n,m}(z,w)\overline{p_{n,m}(z_1,w_1)}}
{(1-z\bar z_1)(1-w \bar w_1)}
\]
is the reproducing kernel for $\clspan\{z^iw^j: (0,0)\leq (i,j)\}$
since
\[
\begin{aligned}
\ip{f}{K_{(z_1,w_1)}} &= \int_{\T^2} \frac{f(z,w)}{p_{n,m}(z,w)}
p_{n,m}(z_1,w_1) \frac{dz dw }{(2\pi i)^2 zw(1-\bar{z} z_1)(1-\bar{w}
  w_1)} \\
&= \frac{f(z_1,w_1)}{p_{n,m}(z_1,w_1)} p_{n,m}(z_1,w_1) = f(z_1,w_1)
\end{aligned}
\]
by the Cauchy integral formula.  On the other hand, 
\begin{equation} \label{prevkernel}
\frac{ \tl{p}_{n,m}(z,w)\overline{\overleftarrow{p_{n,m}}(z_1,w_1)}}
     {(1-z\bar z_1)(1-w \bar w_1)}
\end{equation}
is the reproducing kernel for
\begin{equation} \label{prevspace}
\clspan\{z^iw^j: (0,0)\leq (i,j)\} \ominus \mcH.
\end{equation}
To see this it is enough to show that $\{z^iw^j \tl{p}_{n,m}:
(0,0)\leq (i,j)\}$ is an orthonormal basis for the space
\eqref{prevspace}.  By Lemma \ref{porth}, $z^iw^j \tl{p}_{n,m}$ is in
the space in \eqref{prevspace} for every $i,j\geq 0$ and it is easy to
check that these polynomials form an orthonormal set.  We show that
their span is dense.  

We may write $\tl{p}_{n,m} = c z^n w^m + \text{lower order terms}$
with $c\ne 0$, since $p_{n,m}$ is stable.  Now, let $f$ be in the
space in \eqref{prevspace}.  If $f \perp \tl{p}_{n,m} = cz^nw^m +
\text{lower order terms}$, then since $f$ is already orthogonal to the
``lower order terms'' we see that $f \perp z^nw^m$.  Inductively,
then, we see that assuming $f \perp z^iw^j$ for all $i \leq N$ and
$j\leq M$ but $(i,j)\ne (N,M)$ and assuming $f \perp z^Nw^M \tl{p}_{n,m}$,
we automatically get $f \perp z^Nw^M$ since $f$ will be orthogonal to
the lower order terms in $z^Nw^M\tl{p}_{n,m}$.  Therefore, if $f$ in
\eqref{prevspace} is orthogonal to $\{z^iw^j\tl{p}_{n,m}: i,j\geq 0\}$
there can be no minimal $(i,j)\geq (n,m)$ (in the partial order on
pairs) such that $f$ is not orthogonal to $z^iw^j$. In particular, $f
\perp z^iw^j$ for all $i\geq n$ and $j\geq m$ and by \eqref{prevspace}
$f \perp \mcH$, which forces $f \equiv 0$.

So, $\{z^iw^j\tl{p}_{n,m}: (0,0)\leq (i,j)\}$ is an orthonormal basis
for the space in \eqref{prevspace} while  \eqref{prevkernel} is the
reproducing kernel for this space.

Finally, the reproducing kernel for 
\[
\mcH = \clspan\{z^iw^j: (0,0)\leq (i,j)\} \ominus (\clspan\{z^iw^j:
(0,0)\leq (i,j)\}\ominus \mcH)
\]
is the difference of the reproducing kernels we have just
calculated. Namely,
\[
\frac{p_{n,m}(z,w)\overline{p_{n,m}(z_1,w_1)} -
  \overleftarrow{p_{n,m}}(z,w)\overline{\overleftarrow{p_{n,m}}(z_1,w_1)}}
     {(1-z\bar z_1)(1-w \bar w_1)}.
\]
\end{proof}

\section{The bivariate Christoffel-Darboux formula}

Set 
\begin{align*}
H_1 =& \Span\{z^i w^j: 0\leq i \leq n, 0 \leq j \leq m-1\} \\
&\ominus \Span\{z^i w^j: 0\leq i \leq n-1, 0 \leq j \leq m-1\}
\end{align*}
and
\begin{align*}
\tl{H}_2=&\Span \{ z^iw^j: 0\le i\le n-1, 0\le j\le m\}\\&\ominus \Span \{ z^iw^j, 0\le
i\le n-1, 1\le j\le m\}.
\end{align*}

The two variable Christoffel-Darboux formula is the following.

\begin{theorem}\label{cd}
 Let $p_{n,m}$ be a stable polynomial. Let $K_1$ be the reproducing
 kernel for $H_1$ and let $K_2$ be the reproducing kernel for
 $\tl{H}_2$.  Then
\begin{align*}
&p_{n,m}(z,w)\overline{p_{n,m}(z_1,w_1)} -
  \overleftarrow{p_{n,m}}(z,w)\overline{\overleftarrow{p_{n,m}}(z_1,w_1)}\\&=(1-w\bar
  w_1)K_1(z,w;z_1,w_1)+(1-z\bar z_1)K_2(z,w;z_1,w_1),
\end{align*}
\end{theorem}

\begin{proof}
Set 
\[
\begin{aligned}
\mcH &= \clspan\{z^i w^j: (0,0) \leq (i, j) \ngeq (n,m)\} \\
\mcH_1 &= \clspan\{z^i w^j: 0 \leq i , 0\leq j<m\} \\
\mcH_2 &=  \clspan\{z^i w^j: 0 \leq i<n, 0 \leq j\}
\end{aligned}
\]
and notice that $\mcH_1$ and $\mcH_2$ together span $\mcH$.

Theorem \ref{GW} says
\begin{align}
&\clspan\{z^i a_j(z,w): 0 \leq i, 0\leq j < m\} \\ &= \mcH_1
  \ominus (\mcH_1 \cap \mcH_2) \subset \mcH \ominus
  \mcH_2\label{thmrewritten}
\end{align}
which a fortiori implies
\[
\mcH_1 \ominus (\mcH_1 \cap \mcH_2) = \mcH \ominus \mcH_2.
\]
To see this, suppose $f \in (\mcH \ominus \mcH_2 ) \ominus (\mcH_1
\ominus (\mcH_1 \cap \mcH_2))$.  Then, $f \in \mcH \ominus \mcH_1$.
As $\mcH_1$ and $\mcH_2$ span $\mcH$, such an $f$ must be orthogonal
to all of $\mcH$ and must equal $0$. 

The reproducing kernel for the space 
\[
\mcH \ominus \mcH_2 = \bigoplus_{j\geq 0} z^j H_1 \text{ is } \frac{K_1(z,w;z_1,w_1)}{1-z \bar{z}_1}
\]
from Proposition~\ref{reph}.

If we interchange the roles of $z$ and $w$ in Proposition \ref{reph}
we see that
\[
\frac{K_2(z,w;z_1,w_1)}{1-w\bar{w}_1}
\]
is the reproducing kernel for $\mcH_2$. 

Finally, the reproducing kernel for $\mcH = \mcH_2 \oplus (\mcH\ominus
\mcH_2)$ can be written in two ways.  On the one hand it equals
\[
\frac{p_{n,m}(z,w)\overline{p_{n,m}(z_1,w_1)} -
  \overleftarrow{p_{n,m}}(z,w)\overline{\overleftarrow{p_{n,m}}(z_1,w_1)}}{(1-z\bar
  z_1)(1-w \bar w_1)},
\] 
but on the other it equals
\[
\frac{K_2(z,w;z_1,w_1)}{1-w\bar{w}_1} + \frac{K_1(z,w;z_1,w_1)}{1-z\bar{z}_1}
\]
by the discussion above.  Equating these formulas and multiplying
through by $(1-z\bar{z}_1)(1-w\bar{w}_1)$, yields the desired formula.
\end{proof}

\section{Parametric orthogonal polynomials}

The above results also shed light on the parametric orthogonal
polynomials.   The following proposition shows that the inner products
of $a_0,\dots, a_{m-1}$ with respect to $L^2(d\mu^{\theta}, \T)$ for
the measures parametrized by $z=e^{i\theta} \in \T$
\begin{equation}\label{dmuz}
d\mu^{\theta}(w) = \frac{|d w|}{2\pi |p_{n,m}(e^{i\theta},w)|^2} 
\end{equation}
are \emph{trigonometric polynomials} in $z$. 

\begin{prop} \label{prop:Lnorm}
For fixed $z \in \T$
\begin{equation} \label{Lnorm}
\int_{\T} |L(z,w;\eta)|^2 \frac{|dw|}{2\pi |p_{n,m}(z,w)|^2} = \bar{z}^n L(z,\eta;\eta)
\end{equation}
and as a consequence
\begin{equation} \label{ainnerproducts}
\int_{\T} \overline{a_i(z,w)} a_j(z,w) \frac{|dw|}{2\pi |p_{n,m}(z,w)|^2} =
T_{i,j}(z).
\end{equation}
\end{prop}

\begin{proof} 
For $z \in \T$ the expression
\[
\bar{z}^n L(z,w;\eta) = \frac{p_{n,m}(z,w)\overline{p_{n,m}(z,\eta)} -
  \tl{p}_{n,m}(z,w)\overline{\tl{p}_{n,m}(z,\eta)}}{1-w\bar{\eta}}
\]
is the reproducing kernel/Christoffel-Darboux kernel for polynomials in $w$ of degree at most
$m-1$ with respect to the measure
$|dw|/(2\pi|p_{n,m}(z,w)|^2)$. Indeed, this is one of the
main consequences of the Christoffel-Darboux formula in one variable
(see \cite{hL87} equation (34) or \cite{bS05} Theorem 2.2.7).  It is a
general fact about reproducing kernels $K(w,\eta)=K_{\eta}(w)$ that
\[
||K(\cdot, \eta)||^2 = \ip{K_{\eta}}{K_{\eta}} = K(\eta,\eta).
\]
Using these two observations, \eqref{Lnorm} follows.  Equation
\eqref{ainnerproducts} follows from matching the coefficients of
$\eta^i\bar\eta^j$ in \eqref{Lnorm}.
\end{proof}

Given $T_m(z)$ defined in equation~\eqref{gohsme}, set
$D_i(\theta)$ as the determinant of the $i\times i$ submatrix of
$T_m(e^{i\theta})$ obtained by keeping the first $i$ rows and columns
and set $D_0= 1$ . We now perform the LU decomposition of $T_m$ which
because it is positive definite does not require any pivoting. Set
\begin{equation}\label{paraop}
[\phi^{\theta}_{m-1}(w),\ldots, \phi^{\theta}_0(w)]^T =U(\theta)[w^{m-1},\ldots, 1]^T,
\end{equation}
where $U(\theta)$ is the upper triangular factor obtained from the LU
decomposition of $T_m$ without pivoting. We find:
\begin{prop}\label{pop}
Suppose $p_{n,m}$ is a stable polynomial then
$\{\phi^{\theta}_i(w)\}_{i=0}^{m-1}$ satisfy the relations
\begin{itemize}
 \item $\phi^{\theta}_i(w)$ is a polynomial in $w$ of degree $i$ with
   leading coefficient, $\frac{D_{m-i}(\theta)}{D_{m-i+1}(\theta)}$,
\item $\int_{\T} \phi^{\theta}_i(w)\overline{\phi^{\theta}_j(w)}
  d\mu^{\theta}(w)=\delta_{i,j}\frac{D_{m-i}(\theta)}{D_{m-i+1}(\theta)}$,
\end{itemize}
which uniquely specify the polynomials. The above implies
\[
 \int_{[0,2\pi]^2}e^{i\theta k}
D_{m-j+1}(\theta)\phi^{\theta}_j(e^{i\phi})\overline{\phi^{\theta}_j(e^{i\phi})}\frac{d\theta
d\phi}{(2\pi)^2|p_{n,m}(e^{i\theta},e^{i\phi})|^2}=0,
\ k> n(m-j).
\]
\end{prop}
\begin{proof}
 From the definition of $T_m$ we see that it is the inverse of the
 $m\times m$ moment matrix associated with $d\mu^{\theta}(w)$. The
 first part of the result now follows from the one dimensional theory
 of polynomials orthogonal on the unit circle. The second part follows
 since $z^{n(m-j)}D_{m-j}(\theta)$ is polynomial in $z$.
\end{proof}

\end{document}